\documentclass[12pt]{article}

\usepackage[utf8]{inputenc}
\usepackage{fullpage}
\usepackage{amsmath,amsthm,amssymb}
\usepackage{mathtools}
\usepackage{mathrsfs}

\newtheorem{theorem}{Theorem} 

\newtheorem{proposition}[theorem]{Proposition}

\newtheorem{lemma}[theorem]{Lemma}
\newtheorem{claim}[theorem]{Claim}

\newtheorem{problem}[theorem]{Problem}

\newtheorem{conjecture}[theorem]{Conjecture}\theoremstyle{remark}

\def \bdm {\begin{displaymath}}
	\def \edm {\end{displaymath}}

\newcommand{\N}{\mathbb{N}}

\newcommand{\E}{\mathbb{E}}
\renewcommand{\P}{\mathbb{P}}

\newcommand{\RT}{\text{{\bf RT}}}
\newcommand{\Q}{\text{{\bf Q}}}
\newcommand{\R}{\text{{\bf R}}}
\newcommand{\mcd}{\delta^{\text{cr}}}

\title{Ramsey-Tur\'{a}n problems with small independence numbers}
\author{
	J\'ozsef Balogh\thanks{Department of Mathematics, University of Illinois at Urbana-Champaign, Urbana, Illinois 61801, USA. E-mail: \texttt{\{jobal, cechen4, mccourt4, clchris2\}@illinois.edu}.}  \thanks{Research is partially supported by NSF Grant DMS-1764123, NSF RTG grant DMS 1937241, Arnold O. Beckman Research Award (UIUC Campus Research Board RB 22000), and the Langan Scholar Fund (UIUC). Grace McCourt was partially funded by RTG grant DMS 1937241, while working on this project.}
	\and 
	Ce Chen\footnotemark[1]
	\and 
	Grace McCourt\footnotemark[1]
	\and 
	Cassie Murley\footnotemark[1]
}
\date{}

\begin{document}

\date{}
\maketitle

\begin{abstract}
	Given a graph $H$ and a function $f(n)$, the Ramsey-Tur\'an number $\RT(n,H,f(n))$ is the maximum number of edges in an $n$-vertex $H$-free graph with independence number at most $f(n)$. For $H$ being a small clique, many results about $\RT(n,H,f(n))$ are known and we focus our attention on $H=K_s$ for $s\leq 13$.
	
	By applying Szemer\'edi's Regularity Lemma, the dependent random choice method and some weighted Tur\'an-type results, we prove that these cliques have the so-called \emph{phase transitions} when $f(n)$ is around the inverse function of the off-diagonal Ramsey number of $K_r$ versus a large clique $K_n$ for some $r\leq s$.
\end{abstract}

\section{Introduction}

\subsection{History}\label{history}

Given a graph $H$ and a function $f(n)$, the Ramsey-Tur\'an number, denoted by $\RT(n,H,f(n))$, is the maximum number of edges in an $n$-vertex $H$-free graph with independence number at most $f(n)$. Such problems were studied first by Erd\H os and S\'os~\cite{ErdosSos1}. Denote by $T(n,r)$ the Tur\'an graph, which is the complete $r$-partite graph on $n$ vertices where the size of each part is as equal as possible. Tur\'an theorem~\cite{Turans} states that $T(n,r)$ is the unique $n$-vertex $K_{r+1}$-free graph with the maximum number of edges. Since the independence number of $T(n,r)$ is linear in $n$, substantially different structure appears when $H$ is a clique and $f(n)$ is a sublinear function, i.e. $f(n)=o(n)$. Erd\H os and S\'os~\cite{ErdosSos1} proved that when the forbidden subgraph $H \subseteq G$ is an odd clique, then for $s \geq 1$,\[
\RT(n,K_{2s+1},o(n)) = \frac12\left( \frac{s-1}{s} \right) n^2 + o(n^2).\] The case when $H$ is an even clique proved to be harder, hence more interesting. Szemer\'edi~\cite{Szemeredi} proved $\RT(n,K_4,o(n)) \leq \frac{1}{8}n^2+o(n^2)$, which was the first published application of his Regularity Lemma. Bollob\'as and Erd\H os~\cite{BollErdos} proved $\RT(n,K_4,o(n))=\frac{1}{8}n^2+o(n^2)$ by constructing the so-called Bollob\'as-Erd\H os graph, which was a major surprise at the time, as the function was expected to be $o(n^2)$. Finally, Erd\H os, Hajnal, S\'os and Szemer\'edi~\cite{EHSS} settled the even clique case, showing for $s \geq 2$ \[\RT(n,K_{2s},o(n)) = \frac12 \left( \frac{3s-5}{3s-2} \right)n^2 + o(n^2).\]

\noindent Let \[ \overline{\rho\tau}(H,f)\coloneqq \limsup_{n\to\infty}\frac{\RT(n, H, f(n))}{n^2}\text{\ \ \ and\ \ \ }\underline{\rho\tau}(H,f)\coloneqq \liminf_{n\to\infty}\frac{\RT(n, H, f(n))}{n^2}. \] If $\overline{\rho\tau}(H,f)=\underline{\rho\tau}(H,f)$, then we define $\rho\tau(H,f)\coloneqq \overline{\rho\tau}(H,f)=\underline{\rho\tau}(H,f)$
\footnote{The existence of $\rho\tau(H,f)$ is expected. In the following results, we will abuse the notation a bit, i.e., the upper (lower) bounds of $\rho\tau(H,f)$ are actually upper (lower) bounds of $\overline{\rho\tau}(H,f)$ ($\underline{\rho\tau}(H,f)$). }
and call it the \emph{Ramsey-Tur\'an density of the graph $H$ with respect to the function $f$}. It is easy to see that $\rho\tau(H,f)=c$ if and only if $\RT(n, H, f(n))=cn^2+o(n^2)$. Let $\rho\tau(H,o(f))\coloneqq\lim_{\delta\to 0}\rho\tau(H,\delta f)$. We say that $H$ has a \emph{Ramsey-Tur\'an phase transition} at $f(n)$ if \[\rho\tau(H,f)-\rho\tau(H,o(f))>0.\] Combining Tur\'an Theorem and the above results, we conclude that cliques have their first phase transition at $f(n)=n$. It is natural to investigate whether phase transitions exist for other values of $f(n)$. For cases concerning small cliques, many results are known. We summarize most of them after introducing the necessary definitions and notation. 

The \emph{Ramsey number} $\R(t,m)$ is the minimum integer $n$ such that every $n$-vertex graph contains either a clique $K_t$ or an independent set of size $m$. We use $\Q(t,n)$ to denote the \emph{inverse Ramsey number}, which is the minimum independence number of a $K_t$-free graph on $n$ vertices. In other words, $\R(t,m)=n$ if and only if $\Q(t,n) = m$. We use the function $\Q(t,n)$ usually as follows: If an $n$-vertex graph $G_n$ satisfies $\alpha(G_n)=o(\Q(t,n))$, then every vertex set of $G_n$ of size at least $cn$ spans a $K_t$, for every fixed constant $c$. It follows immediately that we should restrict our attention to $f(n)\geq \Q(t,n)$ if the forbidden graph is $H=K_t$: as by the definition of $\Q(t,n)$, there exists no $n$-vertex $K_t$-free graph with independence number less than $\Q(t,n)$.

In this paper, all logarithms are base $2$ and $w(n)$ is a function going to infinity arbitrarily slowly as $n\to\infty$. Then, $g(n)\coloneqq ne^{-\omega(n)\sqrt{\log n}}$ satisfies that $n^{1-\epsilon}\ll g(n)\ll n$ for every $\epsilon>0$.

It was proved by Shearer~\cite{Shearer}, Pontiveros, Griffiths, Morris~\cite{PGM}, and Bohman, Keevash~\cite{BK-R} that \[
\left(\frac{1}{4}-o(1)\right)\frac{m^2}{\log m}\leq \R(3,m)\leq \left(1+o(1)\right)\frac{m^2}{\log m},
\] which implies \[
\left(\frac{1}{\sqrt{2}}-o(1)\right)\sqrt{n\log n}\leq \Q(3,n)\leq \left(\sqrt{2}+o(1)\right)\sqrt{n\log n}.
\] 

For $t\geq 4$, we do not know the exact order of magnitude of $\R(t,m)$, but there are many well-known conjectures about them, below is one of them.

\begin{conjecture}\label{Ramsey conj}
	For every integer $\ell\geq 3$, there exist $c=c(\ell)>0$ and $N=N(\ell)>0$ such that if $m>N$, then \[
	\R(\ell-1, m)\leq \R(\ell,m)/m^c.\]
\end{conjecture}

Conjecture~\ref{Ramsey conj} holds when $\ell=3, 4$ and is believed to be true for larger $\ell$. Many results in Ramsey-Tur\'an theory are conditional on Conjecture~\ref{Ramsey conj} or its analogues. We will use $\diamond$ to mark such results.\\

\noindent \begin{tabular}{p{68mm}l}
	$\bullet$ \textbf{$K_3$-free:} &\\
	\cite{Mantels}: $\rho\tau(K_3, n/2)=1/4$; & $\rho\tau\left(K_3, o(n)\right)=0$.
\end{tabular}\\\\

\noindent \begin{tabular}{p{68mm}l}
		$\bullet$ \textbf{$K_4$-free:} & \\
		\cite{Turans}: $\rho\tau(K_4, n/3)=1/3$; & \cite{Szemeredi}, \cite{BollErdos}: $\rho\tau\left(K_4, o(n)\right)=1/8$;\\ 
		\cite{Sudakov}: $\rho\tau(K_4, ne^{-\omega(n)\sqrt{\log n}})=0$. 
\end{tabular}\\\\
	
\noindent \begin{tabular}{p{68mm}l}
	$\bullet$ \textbf{$K_5$-free:} & \\
	\cite{Turans}: $\rho\tau(K_5, n/4)=3/8$;& \\
	\cite{ErdosSos1}: $\rho\tau(K_5, o(n))=1/4$; & \cite{BHS}: $\rho\tau(K_5, 2\sqrt{n \log n})=1/4$;\\
	\cite{BHS}: $\rho\tau(K_5, o(\sqrt{n \log n}))=0$.
\end{tabular}\\\\
	
\noindent \begin{tabular}{p{68mm}l}
	$\bullet$ \textbf{$K_6$-free:} & \\
	\cite{Turans}: $\rho\tau(K_6, n/5)=2/5$;& \cite{EHSS}: $\rho\tau(K_6, o(n))=2/7$;\\
	\cite{BHS}: $\rho\tau(K_6, n e^{-\omega(n)\sqrt{\log n}})=1/4$;&  
	\cite{Sudakov}: $\rho\tau(K_6, 2\sqrt{n \log n})=1/4$;\\
	\cite{EHSSS1}: $\rho\tau(K_6, o(\sqrt{n\log n}))\leq 1/6$;& \cite{Sudakov}: $\rho\tau(K_6, \sqrt{n} e^{-\omega(n)\sqrt{\log n}})=0$.
\end{tabular}\\\\
	
\noindent \begin{tabular}{p{68mm}l}
	$\bullet$ \textbf{$K_7$-free:} & \\
	\cite{Turans}: $\rho\tau(K_7, n/6)=5/12$; & \\
	\cite{ErdosSos1}: $\rho\tau(K_7, o(n))= 1/3$; &  \cite{BHS}: $\rho\tau(K_7, 2\sqrt{n \log n})= 1/3$;\\
	\cite{BHS}: $\rho\tau(K_7, o(\sqrt{n \log n}))=1/4$; & \cite{BHS}: $\rho\tau(K_7, \Q(4,n))= 1/4$;\\
	\cite{BHS}: $\rho\tau(K_7, o(\Q(4,n)))=0$.
\end{tabular}\\\\
	
\noindent \begin{tabular}{p{68mm}l}
	$\bullet$ \textbf{$K_8$-free:} &\\
	\cite{Turans}: $\rho\tau(K_8, n/7)=3/7$;& \cite{EHSS}: $\rho\tau(K_8, o(n))=7/20$;\\
	\cite{BHS}: $\rho\tau(K_8, n e^{-\omega(n)\sqrt{\log n}})=1/3$;& \cite{BHS}: $\rho\tau(K_8, 2\sqrt{n \log n})=1/3$;\\
	\cite{KKL}: $\rho\tau(K_8, o(\sqrt{n \log n}))=1/4$;& \cite{BHS}: $\rho\tau(K_8, \Q(4,n))=1/4$;\\
	\cite{BHS}: $\rho\tau(K_8, o(\Q(4,n)))\leq 3/16$;& \cite{BHS}: $\rho\tau(K_8, \Q(4,g(n)))=0$.
\end{tabular}\\\\
	
\noindent \begin{tabular}{p{68mm}l}
	$\bullet$ \textbf{$K_9$-free:} &\\
	\cite{Turans}: $\rho\tau(K_9, n/8)=7/16$;\\
	\cite{ErdosSos1}: $\rho\tau(K_9, o(n))=3/8$; &  \cite{BHS}: $\rho\tau(K_9, 2\sqrt{n \log n})= 3/8$;\\
	($\ast$)\footnotemark[1]: $\rho\tau(K_9, o(\sqrt{n \log n}))\leq 3/10$; &  \cite{BHS}: $\rho\tau(K_9, \Q(3,g(n)))= 1/4$;\\
	($\ast$): ($\diamond$) $\rho\tau(K_9, o(\Q(4,n)))=1/4$; & \cite{BHS}: ($\diamond$) $\rho\tau(K_9, \Q(5,n))=1/4$;\\
	\cite{BHS}: ($\diamond$) $\rho\tau(K_9, o(\Q(5,n)))=0$.
\end{tabular}\\\\

\footnotetext[1]{Results with ($\ast$) will be proved in this paper. We include them here for completeness.}

\noindent \begin{tabular}{p{68mm}l}
	$\bullet$ \textbf{$K_{10}$-free:} &\\
	\cite{Turans}: $\rho\tau(K_{10}, n/9)=4/9$;& \cite{EHSS}: $\rho\tau(K_{10}, o(n))=5/13$;\\
	\cite{BHS}: $\rho\tau(K_{10}, n e^{-\omega(n)\sqrt{\log n}})=3/8$; & \cite{BHS}: $\rho\tau(K_{10}, 2\sqrt{n \log n})= 3/8$;\\
	\cite{BHS}: $\rho\tau(K_{10}, o(\sqrt{n \log n}))=1/3$; & \cite{BHS}: $\rho\tau(K_{10}, \Q(4,n))=1/3$;\\
	($\ast$): ($\diamond$) $\rho\tau(K_{10}, o(\Q(4,n)))=1/4$;& \cite{BHS}: ($\diamond$) $\rho\tau(K_{10}, \Q(5,n))=1/4$;\\
	\cite{BHS}: ($\diamond$) $\rho\tau(K_{10}, o(\Q(5,n)))\leq 1/5$;& \cite{BHS}: ($\diamond$) $\rho\tau(K_{10}, \Q(5,g(n)))=0$.
\end{tabular}\\\\

\noindent \begin{tabular}{p{68mm}l}
	$\bullet$ \textbf{$K_{11}$-free:} &\\
	\cite{Turans}: $\rho\tau(K_{11}, n/10)=9/20$;\\
	\cite{ErdosSos1}: $\rho\tau(K_{11}, o(n))=2/5$; & \cite{BHS}: $\rho\tau(K_{11}, \sqrt{n \log n})=2/5$;\\
	\cite{BHS}: $\rho\tau(K_{11}, o(\sqrt{n \log n}))\leq 7/20$;\\
	\cite{BHS}: $\rho\tau(K_{11}, \sqrt{n} e^{-\omega(n)\sqrt{\log n}})=1/3$; & \cite{BHS}: $\rho\tau(K_{11}, \Q(4,n))=1/3$;\\
	($\ast$): ($\diamond$) $\rho\tau(K_{11}, o(\Q(4,n)))=1/4$; & \cite{BHS}: ($\diamond$) $\rho\tau(K_{11}, \Q(6,n))=1/4$;\\
	\cite{BHS}: ($\diamond$) $\rho\tau(K_{11}, o(\Q(6,n)))=0$.
\end{tabular}\\\\

\noindent \begin{tabular}{p{68mm}l}
	$\bullet$ \textbf{$K_{12}$-free:} &\\
	\cite{Turans}: $\rho\tau(K_{12}, n/11)=5/11$; & \cite{EHSS}: $\rho\tau(K_{12}, o(n))=13/32$;\\
	\cite{BHS}: $\rho\tau(K_{12}, n e^{-\omega(n)\sqrt{\log n}})=2/5$; & \cite{BHS}: $\rho\tau(K_{12}, \sqrt{n \log n})=2/5$;\\
	\cite{BHS}: $\rho\tau(K_{12}, o(\sqrt{n \log n}))\leq 8/22$;\\
	\cite{BHS}: $\rho\tau(K_{12}, \sqrt{n} e^{-\omega(n)\sqrt{\log n}})=1/3$; & \cite{BHS}: $\rho\tau(K_{12}, \Q(4,n))=1/3$;\\
	($\ast$): $\rho\tau(K_{12}, o(\Q(4,n)))\leq 4/13$; \\
	\cite{BHS}: ($\diamond$) $\rho\tau(K_{12},\Q(4,g(n)))=1/4$; & \cite{BHS}: ($\diamond$) $\rho\tau(K_{12}, \Q(6,n))=1/4$;\\
	\cite{BHS}: ($\diamond$) $\rho\tau(K_{12}, o(\Q(6,n)))\leq 5/24$;& \cite{BHS}: ($\diamond$) $\rho\tau(K_{12}, \Q(6,g(n)))=0$.
\end{tabular}\\\\

\noindent \begin{tabular}{p{68mm}l}
	$\bullet$ \textbf{$K_{13}$-free:} &\\
	\cite{Turans}: $\rho\tau(K_{13}, n/12)=11/24$;\\
	\cite{ErdosSos1}: $\rho\tau(K_{13}, o(n))=5/12$; & \cite{BHS}: $\rho\tau(K_{13}, \sqrt{n \log n})=5/12$;\\
	\cite{BHS}: $\rho\tau(K_{13}, o(\sqrt{n \log n}))=3/8$; & \cite{BHS}: $\rho\tau(K_{13}, \Q(4,n))=3/8$;\\
	\cite{BHS}: ($\diamond$) $\rho\tau(K_{13}, o(\Q(4,n)))=1/3$; & \cite{BHS}: ($\diamond$) $\rho\tau(K_{13}, \Q(5,n))=1/3$;\\
	($\ast$): $\rho\tau(K_{13}, o(\Q(5,n)))\leq 4/15$; & ($\ast$): ($\diamond$) $\rho\tau(K_{13}, o(\Q(5,n)))=1/4$;\\
	\cite{BHS}: ($\diamond$) $\rho\tau(K_{13}, \Q(7,n))=1/4$; & \cite{BHS}: ($\diamond$) $\rho\tau(K_{13}, o(\Q(7,n)))=0$.
\end{tabular}\\\\

Although different from the focus of this paper, it is worth mentioning that L\"{u}ders and Reiher \cite{LR} have studied the transition behaviors of cliques at $f(n)=n$ more accurately. For all $s\geq 2$, they proved that if $\delta$ is sufficiently small, then \[ \rho\tau(K_{2s-1},\delta n)=\frac{1}{2}\left(\frac{s-2}{s-1}+\delta\right) \text{\ \ \ and\ \ \ }\rho\tau(K_{2s},\delta n)=\frac{1}{2}\left(\frac{3s-5}{3s-2}+\delta-\delta^2\right).\] When $s=2$, let $G$ be a $K_3$-free graph on $n$ vertices with $\alpha(G) \leq \delta n$, then $e(G)\leq \frac{1}{2}\delta n^2$ since the neighborhood of every vertex is an independent set. Hence, $\rho\tau(K_3,o(n))=0$. Ajtai, Koml\'{o}s and Szemer\'{e}di~\cite{AKS} proved sharper results.

\subsection{Main Results}

Recall that $\Q(3,n)=\Theta(\sqrt{n \log n})$. Kim, Kim and Liu \cite{KKL} determined $\rho\tau(K_8,o(\sqrt{n \log n}))$, which is exactly $\rho\tau(K_8,o(\Q(3,n)))=1/4$.
We extend this result to larger cliques, thus improve the following upper bounds in \cite{BHS}: $\rho\tau(K_9, o(\Q(3,n)))\leq 5/16$, $\rho\tau(K_{10}, o(\Q(4,n)))\leq 5/18$, $\rho\tau(K_{11}, o(\Q(4,n)))\leq 3/10$ and $\rho\tau(K_{12}, o(\Q(4,n)))\leq 7/22$.

\begin{theorem}\label{theorem K_9}
	$\rho\tau(K_9,o(\Q(3,n)))\leq 3/10$.
\end{theorem}

\begin{theorem}\label{theorem K_9,10,11 density}
	$\rho\tau(K_t,o(\Q(4,n)))\leq 1/4$ for $9\leq t\leq 11$.
\end{theorem}

\begin{theorem}\label{theorem K_12 density}
	$\rho\tau(K_{12},o(\Q(4,n)))\leq 4/13$.
\end{theorem}

\begin{theorem}\label{K_13 density}
	$\rho\tau(K_{13},o(\Q(5,n)))\leq 4/15$.
\end{theorem}

If Conjecture~\ref{Ramsey conj} holds for $\ell=5$, then we have a better result for the $K_{13}$-free case, which improves $\rho\tau(K_{13}, o(\Q(5,n)))\leq 7/24$ given in \cite{BHS} under the same assumption.

\begin{theorem}\label{K_13 with assumption density}
	If Conjecture~\ref{Ramsey conj} is true for $\ell=5$, then $\rho\tau(K_{13},o(\Q(5,n)))=1/4$.
\end{theorem}

\noindent\textbf{Remark. }Assuming Conjecture~\ref{Ramsey conj} holds, we explain below that the upper bounds given in Theorems~\ref{theorem K_9,10,11 density} and~\ref{K_13 with assumption density} are tight, which can also be seen by the results listed after Conjecture~\ref{Ramsey conj} in Section~\ref{history}. Moreover, we conjecture that the bounds given in Theorems~\ref{theorem K_9} and~\ref{theorem K_12 density} are best possible.

\medskip
\noindent $\bullet$ \textit{Theorem~\ref{theorem K_9,10,11 density}:} Let $H$ be a $K_5$-free graph on $n/2$ vertices with independence number $\Q(5,n/2)$. The existence of such graphs is guaranteed by Ramsey's theorem and $e(H)=o(n^2)$ by~\cite{BHS}. If Conjecture~\ref{Ramsey conj} holds for $\ell=5$, then $\alpha(H)=\Q(5,n/2)=o(\Q(4,n))$. Let $G$ be obtained from the union of two vertex-disjoint copies of $H$, say $A$ and $B$, by joining every vertex in $A$ to every vertex in $B$. Then, $G$ is $K_9$-free, thus $K_{10}$-free and $K_{11}$-free, with $n^2/4+o(n^2)$ edges and $\alpha(G)=o(\Q(4,n))$. 

\medskip
\noindent $\bullet$  \textit{Theorem~\ref{K_13 with assumption density}:} Let $H$ be a $K_7$-free graph on $n/2$ vertices with independence number $\Q(7,n/2)$. Similarly, the existence of such graphs is guaranteed by Ramsey's theorem and $e(H)=o(n^2)$ by~\cite{BHS}. If Conjecture~\ref{Ramsey conj} holds for $\ell=6$ or for $\ell=7$, then $\alpha(H)=\Q(7,n/2)=o(\Q(5,n))$. Let $G$ be obtained from the union of two vertex-disjoint copies of $H$, say $A$ and $B$, by joining every vertex in $A$ to every vertex in $B$. Then, $G$ is $K_{13}$-free with $n^2/4+o(n^2)$ edges and $\alpha(G)=o(\Q(5,n))$. 

\medskip
\noindent $\bullet$ \textit{Theorem~\ref{theorem K_9}:} We conjecture that $\rho\tau(K_6, o(\sqrt{n\log n}))=1/6$. If it was true, then there exists a $K_6$-free graph $H_1$ on $3n/5$ vertices with independence number $o(\Q(3,n))$ and with $\frac{1}{6}(\frac{3n}{5})^2=3n^2/50$ edges. Let $H_2$ be a $K_4$-free graph on $2n/5$ vertices with independence number $o(\Q(3,n))$ and with $o(n^2)$ edges, the existence of such graphs could be proved with the first moment method. Let $G$ be obtained from the vertex-disjoint union of $H_1$ and $H_2$ by joining every vertex in $H_1$ to every vertex in $H_2$. Then, $G$ is $K_9$-free with $e(G)\leq 3n^2/50+6n^2/25+o(n^2)=3n^2/10+o(n^2)$ and $\alpha(G)=o(\Q(3,n))$.

\medskip
\noindent $\bullet$ \textit{Theorem~\ref{theorem K_12 density}:} We conjecture that $\rho\tau(K_8, o(\Q(4,n)))=3/16$. If it was true, then there exists a $K_8$-free graph $H_1$ on $8n/13$ vertices with independence number $o(\Q(4,n))$ and with $\frac{3}{16}(\frac{8n}{13})^2=12n^2/169$ edges. Let $H_2$ be a $K_5$-free graph on $5n/13$ vertices with independence number $o(\Q(4,n))$ and with $o(n^2)$ edges, the existence of such graphs could be proved with the first moment method. Let $G$ be obtained from the vertex-disjoint union of $H_1$ and $H_2$ by joining every vertex in $H_1$ to every vertex in $H_2$. Then, $G$ is $K_{12}$-free with $e(G)\leq 12n^2/169+40n^2/169+o(n^2)=4n^2/13+o(n^2)$ and $\alpha(G)=o(\Q(4,n))$.

\section{Preliminaries}

\subsection{Definitions and Notation}

In this paper, we will use standard definitions and notation. All graphs considered are simple undirected graphs. Given disjoint sets $A, B \subseteq V(G)$, denote by $N(A,B)$ the common neighborhood of $A$ in $B$. In the case when $A=\{v\}\subseteq V(G)$, we will write $N(v,B)$ for the set of neighbors of $v$ in $B$ and let $d(v,B)\coloneqq|N(v,B)|$. Given a graph $G$ and $U \subseteq V(G)$, the induced subgraph $G[U]$ is the graph whose vertex set is $U$ and whose edge set is spanned by vertices in $U$. If $G[V_1, \ldots, V_p]$ is the induced subgraph of $G$ on the partition of vertices $V_1 \cup \ldots \cup V_p \subseteq G$ where the edges are whose endpoints are in $V_i,V_j$ with $i\neq j$, then $\mcd(V_1,\ldots,V_p) \coloneqq  \min_{\{i,j\} \in \binom{[p]}{2}} \left\{ \min_{v \in V_i} d(v,V_j) \right\}$ is the minimum crossing degree of $G$ with respect to the partition $V_1 \cup \ldots \cup V_p$. We may omit floors and ceilings when they are not essential.


\subsection{Tools}

The following theorem is a corollary of Shearer's bound on $\R(3,n)$.

\begin{theorem}[\cite{KKL}]\label{theorem indep set} 
There exists $k_0 \in \N$ such that for all $k \geq k_0$, every graph with at least $2k^2/\log k$ vertices contains either a triangle or an independent set of size $k$.
\end{theorem}

Although the exact order of magnitude of $\R(4,n)$ is not known, Mattheus and Verstraete~\cite{MV} determined $\R(4,n)$ up to a factor of order $\log^2 n$ very recently.

\begin{theorem}[\cite{MV}]
	\[
	\Omega\left(\frac{n^3}{\log^4 n}\right)\leq \R(4,n) \leq O\left(\frac{n^3}{\log^2n}\right).
	\]
\end{theorem}

\noindent Therefore, \[ \Omega\left(n^{\frac{1}{3}}(\log n)^{\frac{2}{3}}\right)\leq \Q(4,n) \leq O\left(n^{\frac{1}{3}}(\log n)^{\frac{4}{3}}\right).
\]

For $t\geq 5$, the following is known.

\begin{theorem}[\cite{AKS}, \cite{BK}] For $t\geq 5$, we have
	\[\Omega\left(\frac{n^\frac{t+1}{2}}{(\log n)^{\frac{t+1}{2}-\frac{1}{t-2}}}\right)\leq \R(t,n)\leq O\left(\frac{n^{t-1}}{(\log n)^{t-2}}\right).\]
\end{theorem}

\noindent In particular, \[
\Omega\left(n^{\frac{1}{4}}(\log n)^{\frac{3}{4}}\right)\leq \Q(5,n)\leq \left(n^{\frac{1}{3}}(\log n)^{\frac{8}{9}}\right).\] 

For disjoint vertex sets $A$ and $B$ in $G$, denote by $d_G(A,B)\coloneqq \frac{e(G[A,B])}{|A||B|}$ the density of the pair $(A,B)$ in $G$. For $\epsilon>0$, we say that a pair $(A,B)$ is \emph{$\epsilon$-regular} if for every $A'\subseteq A$ and $B'\subseteq B$ such that $|A'|\geq \epsilon|A|$ and $|B'|\geq \epsilon|B|$, we have $|d_G(A',B')-d_G(A,B)|\leq \epsilon$. If additionally $d_G(A,B)\geq \gamma$, then we say that $(A,B)$ is \emph{($\epsilon,\gamma)$-regular}. A partition $V_1\cup\cdots\cup V_m$ of $V(G)$ is \emph{$\epsilon$-regular} if it is an equipartition and all but at most $\epsilon m^2$ pairs $(V_i,V_j)$ are $\epsilon$-regular.

\begin{lemma}[Szemer\'edi's Regularity Lemma, \cite{regularity}]\label{regularity lemma}
Suppose $0 < 1/M' \ll \epsilon$, $1/M \ll 1$ and $n \geq M$. For every $n$-vertex graph $G$ there exists an $\epsilon$-regular partition $V(G) = V_1 \cup \cdots \cup V_m$ with $M \leq m \leq M'$.
\end{lemma}

\begin{lemma}[Slicing Lemma, \cite{Szemeredi}]\label{slicing lemma} 
Let $\epsilon < \alpha, \gamma, 1/2$. Suppose that $(A,B)$ is an $(\epsilon, \gamma)$-regular pair in a graph $G$. If $A' \subseteq A$ and $B' \subseteq B$ satisfies $|A'| \geq \alpha |A|$ and $|B'| \geq \alpha |B|$, then $(A', B')$ is an $(\epsilon', \gamma-\epsilon)$-regular pair in $G$, where $\epsilon' \coloneqq  \max\{\epsilon/\alpha, 2\epsilon\}.$
\end{lemma}

Let $\epsilon, \gamma>0$. For a given graph $G$ with partition $V_1\cup\cdots\cup V_m$, we define the cluster graph $R\coloneqq R(\epsilon, \gamma)$ as follows: $V(R)=[m]$; $ij\in E(R)$ if and only if $(V_i, V_j)$ is an $(\epsilon,\gamma)$-regular pair in $G$. We use the following lemma to derive Tur\'an-type properties of the cluster graph.

\begin{lemma}[\cite{BHS}]\label{lemma to derive K_4 free} 
	Let $G$ be an $n$-vertex graph with $\alpha(G) < \mathbf{Q}(p, n2^{-\omega(n)\log^{1-1/q}n})$. We apply Szemer\'edi's Regularity Lemma to $G$ to obtain an $\epsilon$-regular partition of $V(G)$ and the corresponding cluster graph $R=R(\epsilon,\gamma)$, where $\epsilon\ll \gamma$. If there exists a $K_q$ in $R$, then we can find a $K_{pq}$ in $G$.
\end{lemma}


We also need the following dependent random choice type of lemma, which is a generalization of Lemma 3.1 in \cite{KKL}.

\begin{lemma}\label{lemma 3.1 4-partite generalized}
	Let $k\geq 2$ be a fixed integer. Suppose $G = Z_1 \cup \cdots \cup Z_k$ is a $k$-partite graph with $|Z_i| = n$ for each $i \in [k]$. Let $0 < \gamma < 1$, $c\gg\frac{1}{\log n}$  and $2\leq t \in \N$ be fixed real numbers. If $|N(v, Z_i)| \geq \gamma n$ for every $v \in Z_k$ and $i \in [k-1]$, then there exists $S \subseteq Z_k$ of size $|S| = \frac{1}{2}n^{1-c}$ such that every $t$-tuple of vertices $T \in \binom{S}{t}$ satisfies $|N(T,Z_i)| \geq \gamma^{\frac{2(k-1)(t+1)}{c}}n$ for each $i \in [k-1]$. 
\end{lemma}

\begin{proof}
For each $i \in [k-1]$, let $Q_i$ be a set of vertices in $Z_i$ chosen uniformly at random with repetition such that $|Q_i| \coloneqq  q \coloneqq  -\frac{c}{2(k-1)}\log_{\gamma}n$. Call $T \in \binom{Z_k}{t}$ a {\it bad} $t$-tuple if there exists $i \in [k-1]$ such that $|N(T,Z_i)| < \gamma^a n$, where $a\coloneqq  \frac{2(k-1)(t+1)}{c}$. Let $S' \coloneqq  N(\cup_{i=1}^{k-1}Q_i, Z_k)$. Define a random variable $X$ to be the number of bad $t$-tuples $T$ with $T\subseteq S'$. For every bad $t$-tuple $T \in \binom{Z_k}{t}$, we have
\[
\P[T \subseteq S'] = \P\left[\bigcup\limits_{i=1}^{k-1}Q_i \subseteq N(T)\right] = \prod_{i=1}^{k-1} \left(\frac{|N(T,Z_i)|}{|Z_i|}\right)^q \leq \left( \frac{\gamma^a n}{n} \right)^q= \gamma^{aq}.
\]

\noindent By linearity of expectation, we have \[
 \E[X] \leq \binom{n}{t} \gamma^{aq} \leq n^t\gamma^{aq}.\] We also have
  \begin{equation*}
 	\begin{aligned}
 		\E [ |S'|] &= \sum_{v \in Z_1} \P[v \in S'] = \sum_{v \in Z_1} \P\left[\bigcup\limits_{i=1}^{k-1}Q_i \subseteq N(v)\right]\\
 		&= \sum_{v \in Z_1} \prod_{i=1}^{k-1} \left(\frac{|N(v,Z_i)|}{|Z_i|}\right)^q \geq n\left( \frac{\gamma n}{n} \right)^{(k-1)q}= n\gamma^{(k-1)q}.
 	\end{aligned}
 \end{equation*}

\noindent Therefore,
\[\E[|S'|-X] \geq n\gamma^{(k-1)q}-n^t\gamma^{aq} = n^{1-c/2} - n^{-1} \geq \frac{1}{2}n^{1-c/2} \geq \frac12 n^{1-c},\] which implies that there exist $Q_1, \ldots, Q_{k-1}$ such that $|S'|-X\geq \frac{1}{2}n^{1-c}$. Let $S \subseteq Z_k$ be the set obtained by deleting one vertex from every bad $t$-tuple in $S'$, then $S$ will satisfy the conclusion of the lemma.
\end{proof}

\subsection{Weighted Tur\'an-type Results}

Among others, we utilize a series of weighted Tur\'an-type results to analyze the properties of the cluster graph.



Let $G$ be a graph. The \emph{standard clique edge-weighting} is an assignment $w$ of weights to the edges of $G$ as follows. Let $e$ be an edge and $r$ be the order of the largest clique containing $e$ in $G$. Then we define the weight of the edge $e$ as \[
w(e) \coloneqq  \frac{r}{2(r-1)}.
\]
We extend the definition of the weight function $w$ to $G$: \[
w(G) \coloneqq  \sum_{e\in E(G)} w(e).
\]
Notice that the weights are defined such that
for every $r \geq 2$, 
\[
\lim_{n \to \infty}\frac{w(T(n,r))}{n^2} = \frac{1}{4}.
\]


\begin{theorem}[\cite{CliqueWeighting}, \cite{CliqueWeighting2}]\label{cliqueweighting}
	Let $G$ be an $n$-vertex graph and $w$ be the standard clique edge-weighting. Then 
	\[
	w(G) \leq \frac{n^2}{4}.
	\]
	Equality holds when $n$ is a multiple of some $r$ and $G=T(n,r)$ is the Tur\'an graph.
\end{theorem}


The following variations of the standard clique edge-weighting theorem were discussed in~\cite{BL}.

\begin{theorem}[Theorem 3.4, \cite{BL}]\label{K_5-free_chubby}
	Let $G$ be an $n$-vertex $K_4$-free graph with a weight function $w$ on $E(G)$ as follows: if an edge $e\in E(G)$ is contained in some triangle, then let $w(e)\coloneqq 4/5$; otherwise, let $w(e)\coloneqq 1$. Then \[
	w(G) \leq \left(\frac{4}{15}+o(1)\right)n^2.
	\] 	Moreover, for every $\varepsilon > 0$ if $n$ is sufficiently large and $w(G) = \left(\frac{4}{15} +o(1)\right)n^2$, then $G$ is in edit distance at most  $\varepsilon n^2$ from $T(n,3)$.
\end{theorem}

A triangle in $G$ is called \emph{a-heavy} if for every edge $e$ of it $w(e)>a$. A triangle in $G$ is called \emph{b-chubby} if for some edge $e$ of it $w(e)>b$. The following result is an immediate corollary of Theorem 3.2 in~\cite{BL}. 

\begin{theorem}[\cite{BL}]\label{K_4-free-heavy}
	Let $G$ be an $n$-vertex $K_4$-free graph with a weight function $w: E(G)\to [0,1]$. Let $a\in[0,1]$ be fixed. Suppose that $G$ contains no $a$-heavy triangle.\\
	(i) If $a=2/3$, then \[w(G)\leq \left(\frac{3}{10}+o(1)\right)n^2.\] 
	(ii) If $a=3/4$, then \[w(G)\leq \left(\frac{4}{13}+o(1)\right)n^2.\] 
\end{theorem}

\subsection{Proof Idea of the Main Results}

Let $9\leq t\leq 13$ be the size of the clique we want to forbid and $f(n)$ be the corresponding inverse Ramsey number depending on the choice of $t$, i.e., for $t=9$, $f(n)=\Q(3,n)$ or $\Q(4,n)$; for $10\leq t\leq 12$, $f(n)=\Q(4,n)$; while for $t=13$, $f(n)=\Q(5,n)$. Fix $\gamma>0$ and let \begin{equation}\label{parameter}
	0 < 1/n_0\leq 1/n \ll \delta < 1/M' \ll \epsilon \ll \gamma \ll 1.
\end{equation}
 Let $G$ be an $n$-vertex $K_t$-free graph with $\alpha(G)\leq \delta f(n)$. By applying Lemma~\ref{regularity lemma} to $G$,
we obtain an $\epsilon$-regular partition $V(G)=V_1 \cup \cdots \cup V_m$ with $1/\epsilon\leq m\leq M'$. Let $R\coloneqq  R(\epsilon, \gamma/2)$ be the corresponding cluster graph on $m$ vertices. By Lemma~\ref{lemma to derive K_4 free}, we know that $R$ contains no large clique. More specifically: By the $p=2$ version of Lemma~\ref{lemma to derive K_4 free}, $R$ is $K_5$-free when $t=9, f(n)=\Q(3,n)$; By the $p=3$ version of Lemma~\ref{lemma to derive K_4 free}, $R$ is $K_4$-free when $9\leq t\leq 12, f(n)=\Q(4,n)$, and $R$ is $K_5$-free when $t=13, f(n)=\Q(5,n)$. Note that each edge $ij\in E(R)$ corresponds to at most $d_G(V_i,V_j)(\frac{n}{m})^2$ edges in $G$. The number of the rest of the edges of $G$, which are exactly those not corresponding to $E(R)$, is at most \begin{equation}\label{nonedge of R}
	\epsilon m^2\left(\frac{n}{m}\right)^2+\frac{\gamma}{2}\left(\frac{n}{m}\right)^2\binom{m}{2}+\binom{n/m}{2}m\leq \epsilon n^2+\frac{\gamma}{4}n^2+\frac{1}{2m}n^2\leq \frac{\gamma}{3}n^2,
\end{equation} which is small. Instead of computing $e(G)$ directly as in \cite{KKL}, we apply the weighted Tur\'an-type results from Section 2.3 to obtain an upper bound on $e(R)$, thus obtain an upper bound on the corresponding number of edges in $G$, which makes up most of $E(G)$.

\section{Proofs of the Main Results}\label{proofofmain}

\subsection{$K_9$-free: Proof of Theorem~\ref{theorem K_9}}\label{9pf}

Let all the parameters be as in (\ref{parameter}). Let $G$ be an $n$-vertex $K_9$-free graph with $\alpha(G)\leq \delta \Q(3,n)$. To prove Theorem~\ref{theorem K_9}, it suffices to prove that \[
e(G)\leq \frac{3}{10}n^2+\gamma n^2.\] Let $R\coloneqq R(\epsilon, \gamma/2)$ be the corresponding cluster graph on $m$ vertices. By the $p=2$ version of Lemma~\ref{lemma to derive K_4 free}, $R$ is $K_5$-free. In fact, we can prove that $R$ contains no $K_4$.

\begin{claim}\label{cluster_k4-free}
	$R$ is $K_4$-free.
\end{claim}

\begin{proof}
	Suppose that $\{1,2,3,4\}$ spans a $K_4$ in $R$. Then, $(V_i, V_j)$ is $\epsilon$-regular with $d(V_i,V_j) \geq \gamma/2$ for every pair $\{i,j\}\in\binom{[4]}{2}$. For each $i \in [4]$, there exists a subset $V_i^* \subseteq V_i$ such that $|V_i^*| = (1-3\epsilon)|V_i|$ and $\delta^{cr}(G[V_1^*,V_2^*,V_3^*, V_4^*])\geq \gamma|V_i^*|/4$. Applying Lemma~\ref{lemma 3.1 4-partite generalized} to $G[V_1^*,V_2^*,V_3^*,V_4^*]$ with $k=4, c=1/3$ and $t=2$ gives us a set $S \subseteq V_1^*$ of size $\frac{1}{2}|V_1^*|^{2/3}\geq \frac{1}{3}(\frac{n}{m})^{2/3}$ such that every $P \in \binom{S}{2}$ satisfies $|N(P,V_i^*)| \geq (\frac{\gamma}{4})^{54}|V_i^*|\geq \gamma^{56}\frac{n}{m}$ for each $i \in \{2,3,4\}$. Recall that $\alpha(G)\leq \delta \Q(3,n)$ and $\Q(3,n)=\Theta (\sqrt{n\log n})$. Since $\frac{1}{3}(\frac{n}{m})^{2/3}>\alpha(G)$, the set $S$ contains an edge $uv\in E(G)$ with $|N(\{u,v\},V_i^*)|\geq \gamma^{56}\frac{n}{m}$ for each $i \in \{2,3,4\}$. 
	
	By applying Lemma~\ref{slicing lemma} and deleting all vertices of low degree if necessary, we could get subsets $V_i' \subseteq N(\{u,v\},V_i^*)$ for $i \in \{2,3,4\}$ satisfying that $|V'_2|=|V'_3|=|V'_4| \geq \gamma^{60}\frac{n}{m}$, $\delta^{cr}(G[V'_2,V'_3,V'_4]) \geq  \gamma|V'_i|/5$, and $(V'_i,V'_j)$ is $(\sqrt{\epsilon}, \gamma/4)$-regular for every pair $\{i,j\}\in\binom{\{2,3,4\}}{2}$. We apply Lemma~\ref{lemma 3.1 4-partite generalized} to $G[V'_2,V'_3,V'_4]$ with $k=3, c=1/3$ and $t=2$. This gives us a set $S' \subseteq V'_2$ of size $\frac{1}{2}|V'_2|^{2/3}\geq \gamma^{41}(\frac{n}{m})^{2/3}$ such that every $P \in \binom{S'}{2}$ satisfies $|N(P,V'_i)| \geq (\frac{\gamma}{5})^{36}|V'_i|\geq \gamma^{97}\frac{n}{m}$ for each $i \in \{3,4\}$. Since $\gamma^{41}(\frac{n}{m})^{2/3}>\alpha(G)$, the set $S'$ contains an edge $xy\in E(G)$ with $|N(\{x,y\},V'_i)|\geq \gamma^{97}\frac{n}{m}$ for each $i \in \{3,4\}$. 
	
	Again, by applying Lemma~\ref{slicing lemma} and deleting all vertices of low degree if necessary, we could get $V_i'' \subseteq N(\{x,y\},V'_i)$ for $i \in \{3,4\}$ such that $|V''_3|=|V''_4| \geq \gamma^{99}\frac{n}{m}$, $\delta(G[V''_3,V''_4])\geq \gamma|V''_i|/6$ and $(V''_3,V''_4)$ is $(\epsilon^{1/4}, \gamma/5)$-regular. We apply Lemma~\ref{lemma 3.1 4-partite generalized} once more to $G[V_3'',V_4'']$ with $k=2, c=1/3$ and $t=2$. This gives us a set $S''\subseteq V''_3$ of size $\frac{1}{2}|V''_3|^{2/3}\geq \gamma^{67}(\frac{n}{m})^{2/3}$ such that every $P\in\binom{S''}{2}$ satisfies $|N(P,V''_4)| \geq (\frac{\gamma}{6})^{18}|V''_4|\geq \gamma^{118}\frac{n}{m}$. Again, $S''$ contains an edge $zw\in E(G)$ with $|N(\{z,w\},V''_4)|\geq \gamma^{118}\frac{n}{m}$ since $\gamma^{67}(\frac{n}{m})^{2/3}>\alpha(G)$. Note that $\Q(3,\gamma^{118}\frac{n}{m})>\delta\Q(3,n)\geq \alpha(G)$. Therefore, $|N(\{z,w\},V''_4)|$ contains a $K_3$, which together with $uv, xy$ and $zw$ forms a $K_9$ in $G$, a contradiction.
\end{proof}

Now we analyze the triangles in $R$. Recall that a triangle $ijk$ in $R$ is \textit{$(2/3+\gamma)$-heavy} if $d_G(V_{i'}, V_{j'}) > 2/3 + \gamma$ for all pairs $i'j' \in \binom{\{i,j,k\}}{2}$.

\begin{claim}\label{modified Claim 3.2a Q3}
	No triangle in $R$ is $(2/3+\gamma)$-heavy.
\end{claim}

\begin{proof}
	Suppose that $\{1,2,3\}$ spans a $(2/3+\gamma)$-heavy triangle in $R$. Then all pairs $(V_i,V_j)$ with $ij\in\binom{[3]}{2}$ are $\epsilon$-regular with $d_G(V_i,V_j)>2/3+\gamma$, so there exist $V_i^* \subseteq V_i$ for every $i \in [3]$ such that $|V_i^*| = (1-2\epsilon)|V_i|$ and $\delta^{cr}(G[V_1^*,V_2^*,V_3^*])\geq (2/3+\gamma/2)|V_i^*|$. We will work with these sets with high minimum crossing degree.

	We claim that $V_1^*$ contains a triangle, otherwise, by Theorem~\ref{theorem indep set}, $V_1^*$ would contain an independent set of size at least \begin{equation*}
		\frac{1}{2}\sqrt{|V_1^*|\log|V_1^*|}\geq\frac{1}{2} \sqrt{(1-2\epsilon)\frac{n}{m}\log\left((1-2\epsilon)\frac{n}{m}\right)}>\frac{1}{m}\sqrt{n\log n}\geq \alpha(G),
	\end{equation*} a contradiction. Suppose that $S$ spans a triangle in $V_1^*$, then each vertex in $S$ has at least $(2/3+\gamma/2)|V_i^*|$ neighbors in $V_i^*$ for $i \in \{2,3\}$ because $\delta^{cr}(G[V_1^*,V_2^*,V_3^*])\geq (2/3+\gamma/2)|V_i^*|$. Then the size of the intersection of the three neighborhoods, which is $|N(S,V_i^*)|$, is at least $\frac{3}{2}\gamma|V_i^*|\geq \gamma^2\frac{n}{m}$ for $i\in\{2,3\}$.
	
	Using Lemma~\ref{slicing lemma} and again deleting vertices of low degree, we could obtain $V_i' \subseteq N(S,V_i^*)$ for $i \in \{2,3\}$ satisfying that $|V'_2|=|V'_3| \geq \gamma^4\frac{n}{m}$, $\delta(G[V'_2,V'_3]) \geq (2/3 + \gamma/4)|V'_i|$ and $(V'_2,V'_3)$ is $(\sqrt{\epsilon}, 2/3+\gamma/3)$-regular. Again, by the low independence number condition, $V_2'$ must contain a triangle.
	
	Let $T$ be a triangle in $V_2'$, then we have $|N(T,V_3')| \geq \frac{3}{4}\gamma|V_3'| \geq \gamma^6\frac{n}{m}$. Again, by the low independence number condition, $N(T,V_3')$ contains a triangle, which together with $S$ and $T$ forms a $K_9$, a contradiction.
\end{proof}

Define a weight function on $E(R)$ as follows: If $ij\in E(R)$ satisfies $d_G(V_i,V_j)\in (2/3, 2/3+\gamma)$, then let $w(ij)\coloneqq 2/3$, otherwise, let $w(ij)\coloneqq d_G(V_i,V_j)$. By Claim~\ref{modified Claim 3.2a Q3} and Theorem~\ref{K_4-free-heavy}, \[
w(R)\leq \left(\frac{3}{10}+\frac{\gamma}{6}\right)m^2,
\] as we can always assume $M$, thus $m$, to be sufficiently large at the very first step. Let $a$ be the number of edges $ij\in E(R)$ with $d_G(V_i,V_j)\in (2/3, 2/3+\gamma)$. Then $a\leq m^2/2$ and every such edge $ij$ contributes $d_G(V_i,V_j)(n/m)^2\leq (2/3+\gamma)(n/m)^2=(w(ij)+\gamma)(n/m)^2$ edges to $G$. Notice that every other edge $ij\in E(R)$ contributes exactly $d_G(V_i,V_j)(n/m)^2=w(ij)(n/m)^2$ edges to $G$. Therefore, $E(R)$ contributes \[ \sum_{ij\in E(R)}d_G(V_i,V_j)\left(\frac{n}{m}\right)^2\leq 
w(R)\left(\frac{n}{m}\right)^2+a\gamma \left(\frac{n}{m}\right)^2\leq \left(\frac{3}{10}+\frac{2\gamma}{3}\right)n^2 \] edges to $G$. Combining with (\ref{nonedge of R}), \[
e(G)\leq \left(\frac{3}{10}+\gamma\right) n^2\] and we completed the proof of Theorem~\ref{theorem K_9}.

\subsection{$K_t$-free for $t=9, 10, 11$: Proof of Theorem~\ref{theorem K_9,10,11 density}}\label{9-11pf}

Let $t\in\{9,10,11\}$ and all the parameters be as in (\ref{parameter}). Let $G$ be an $n$-vertex $K_t$-free graph with $n\geq n_0$ and $\alpha(G)\leq \delta \Q(4,n)$. To prove Theorem~\ref{theorem K_9,10,11 density}, it suffices to prove that \[
e(G)\leq \frac{1}{4}n^2+\gamma n^2.\] Let $R\coloneqq R(\epsilon, \gamma/2)$ be the corresponding cluster graph. By the $p=3$ version of Lemma~\ref{lemma to derive K_4 free}, $R$ is $K_4$-free. Instead of considering \emph{$(2/3+\gamma)$-heavy} triangles in $R$, we will use \emph{$(3/4+\gamma)$-chubby} triangles: Recall that a triangle $ijk$ in $R$ is \textit{$(3/4+\gamma)$-chubby} if $d_G(V_{i'}, V_{j'}) > 3/4 + \gamma$ for some $i'j' \in \binom{\{i,j,k\}}{2}$.

\begin{claim}\label{modified Claim 3.2a}
No triangle in $R$ is $(3/4+\gamma)$-chubby.
\end{claim}

\begin{proof}
Suppose that $\{1,2,3\}$ spans a $(3/4+\gamma)$-chubby triangle in $R$ with $d(V_2,V_3)>3/4 + \gamma$. Since all pairs $(V_i,V_j)$ for $ij\in\binom{[3]}{2}$ are $\epsilon$-regular, we have that for each $i \in [3]$, there exists $V_i^* \subseteq V_i$ such that $|V_i^*| = (1-2\epsilon)|V_i|$ and $\delta^{cr}(G[V_1^*,V_2^*,V_3^*])\geq \gamma|V_i^*|/3$. We will find a $K_{11}$ in these sets with high minimum crossing degree, using the dependent random choice method.

Apply Lemma~\ref{lemma 3.1 4-partite generalized} to $G[V_1^*,V_2^*,V_3^*]$ with $k=3, c=1/5$ and $t=3$. This gives us a set $S \subseteq V_1^*$ of size $\frac{1}{2}|V_1^*|^{4/5}\geq\frac{1}{3}(\frac{n}{m})^{4/5}$ such that every triple $P \in \binom{S}{3}$ satisfies $|N(P,V_i^*)| \geq (\frac{\gamma}{3})^{80}|V_i^*|\geq \gamma^{82}\frac{n}{m}$ for each $i \in \{2,3\}$. Recall that $\alpha(G)\leq \delta \Q(4,n)$ and $\Q(4,n)=O(n^{2/5})$, so $\frac{1}{3}(\frac{n}{m})^{4/5}>(\alpha(G))^2$. Therefore, $S$ contains a triangle $uvw$ with $|N(\{u,v,w\},V_i^*)|\geq \gamma^{82}\frac{n}{m}$ for each $i \in \{2,3\}$.

By applying Lemma~\ref{slicing lemma} and deleting all vertices of low degree if necessary, we could get subsets $V_i' \subseteq N(\{u,v,w\},V_i^*)$ for $i \in \{2,3\}$ such that $|V'_2|=|V'_3| \geq \gamma^{84}\frac{n}{m}$, $\delta(G[V'_2,V'_3]) \geq (3/4 + \gamma/5)|V'_i|$ and $(V'_2,V'_3)$ is $(\sqrt{\epsilon}, 3/4+\gamma/4)$-regular. We claim that $V'_2$ contains a $K_4$. Otherwise, by the definition of inverse Ramsey number, we could always choose $\delta$ small enough such that there exists an independent set of size at least 
\[ \Q\left(4, |V'_2|\right) \geq 
\Q\left(4,\frac{\gamma^{84}}{m}n\right)> \delta \Q(4,n) \geq \alpha(G),
\] a contradiction.

Let $T$ be a $K_4$ in $V_2'$, then $|N(T,V_3')| \geq \frac{4\gamma}{5}|V_3'| \geq \gamma^{86}\frac{n}{m}$ since $\delta(G[V'_2,V'_3]) \geq (3/4 + \gamma/5)|V'_i|$. Again, by the low independence number condition, $N(T,V_3')$ contains a $K_4$, which together with $T$ and $uvw$ forms a $K_{11}\supseteq K_t$, a contradiction.
\end{proof}


	
Let $a$ be the number of edges in $R$ contained in some triangle, then $a\leq m^2/2$. Let $b\coloneqq e(R)-a$. Recall that $R$ is $K_4$-free. By Theorem~\ref{cliqueweighting},
\[ a\cdot\frac{3}{4}+b\leq \left(\frac{1}{4}+\frac{\gamma}{6}\right)m^2.
\]
By Claim~\ref{modified Claim 3.2a}, every edge $ij\in E(R)$ contained in some triangle satisfies that $d_G(V_i,V_j)\leq 3/4+\gamma$. Therefore, $E(R)$ contributes \begin{equation*}
	\begin{aligned}
			\sum_{ij\in E(R)}d_G(V_i,V_j)\left(\frac{n}{m}\right)^2&\leq a\cdot\left(\frac{3}{4}+\gamma\right)\left(\frac{n}{m}\right)^2+b\cdot\left(\frac{n}{m}\right)^2\\
			&\leq\left(\frac{1}{4}+\frac{\gamma}{6}\right)n^2+a\gamma\left(\frac{n}{m}\right)^2\leq \left(\frac{1}{4}+\frac{2\gamma}{3}\right)n^2
	\end{aligned}
	\end{equation*} edges to $G$. Combining with (\ref{nonedge of R}), \[
e(G)\leq \left(\frac{1}{4}+\gamma\right) n^2\] as desired.

\subsection{$K_{12}$-free: Proof of Theorem~\ref{theorem K_12 density}}\label{12pf}

The proof for this case is similar to the one in Section 3.1, with the use of \emph{$(3/4+\gamma)$-heavy} triangles instead of \emph{$(2/3+\gamma)$-heavy} triangles. Let all the parameters be as in (\ref{parameter}). Let $G$ be an $n$-vertex $K_{12}$-free graph with $\alpha(G)\leq \delta \Q(4,n)$ and $R\coloneqq R(\epsilon, \gamma/2)$ be the corresponding cluster graph. Note that $R$ is $K_4$-free by the $p=3$ version of Lemma~\ref{lemma to derive K_4 free}. Our aim now is to prove that \[
e(G)\leq \frac{4}{13}n^2+\gamma n^2.\]

\begin{claim}\label{Claim-heavy-triangle}
No triangle in $R$ is $(3/4+\gamma)$-heavy.
\end{claim}

\begin{proof}
Suppose that $\{1,2,3\}$ spans a $(3/4+\gamma)$-heavy triangle in $R$. Then all pairs $(V_i,V_j)$ with $ij\in\binom{[3]}{2}$ are $\epsilon$-regular with $d_G(V_i,V_j)>3/4+\gamma$, so there exist subsets $V_i^* \subseteq V_i$ for every $i \in [3]$ such that $|V_i^*| = (1-2\epsilon)|V_i|$ and $\delta^{cr}(G[V_1^*,V_2^*,V_3^*])\geq (3/4+\gamma/2)|V_i^*|$. We will still use these sets with high minimum crossing degree.

We claim that $V_1^*$ contains a $K_4$. Otherwise, by the definition of inverse Ramsey number, $V_1^*$ contains an independent set of size at least \begin{equation*}
\Q(4,|V_1^*|)=\Q\left(4,\frac{1-2\epsilon}{m}n\right)>\delta \Q(4,n)\geq \alpha(G),
\end{equation*} a contradiction. Suppose that $S$ spans a $K_4$ in $V_1^*$. Then the size of  $|N(S,V_i^*)|$ is at least $2\gamma|V_i^*|\geq \gamma^2\frac{n}{m}$ for $i\in\{2,3\}$ since $\delta^{cr}(G[V_1^*,V_2^*,V_3^*])\geq (3/4+\gamma/2)|V_i^*|$.

Using Lemma~\ref{slicing lemma} and again deleting vertices of low degree, we obtain $V_i' \subseteq N(S,V_i^*)$ for $i \in \{2,3\}$ satisfying that $|V'_2|=|V'_3| \geq \gamma^4\frac{n}{m}$, $\delta(G[V'_2,V'_3]) \geq (3/4 + \gamma/4)|V'_i|$ and $(V'_2,V'_3)$ is $(\sqrt{\epsilon}, 3/4+\gamma/3)$-regular. Again, by the low independence number condition, $V_2'$ must contain a $K_4$.

Let $T$ be a $K_4$ in $V_2'$, then we have $|N(T,V_3')| \geq \gamma|V_3'| \geq \gamma^5\frac{n}{m}$. Again, by the low independence number condition, $N(T,V_3')$ contains a $K_4$, which together with $S$ and $T$ forms a $K_{12}$, a contradiction.
\end{proof}

Define a weight function on $E(R)$ as follows: If $ij\in E(R)$ satisfies $d_G(V_i,V_j)\in (3/4, 3/4+\gamma]$, then $w(ij)\coloneqq 3/4$, otherwise, let $w(ij)\coloneqq d_G(V_i,V_j)$. By Claim~\ref{Claim-heavy-triangle} and Theorem~\ref{K_4-free-heavy}, \[
w(R)\leq \left(\frac{4}{13}+\frac{\gamma}{6}\right)m^2.
\] Let $a$ be the number of edges $ij\in E(R)$ with $d_G(V_i,V_j)\in (3/4, 3/4+\gamma)$. Then $a\leq m^2/2$ and every such edge $ij$ contributes $d_G(V_i,V_j)(n/m)^2\leq (3/4+\gamma)(n/m)^2=(w(ij)+\gamma)(n/m)^2$ edges to $G$. Notice that every other edge $ij\in E(R)$ contributes exactly $d_G(V_i,V_j)(n/m)^2=w(ij)(n/m)^2$ edges to $G$. Therefore, $E(R)$ contributes \[ \sum_{ij\in E(R)}d_G(V_i,V_j)\left(\frac{n}{m}\right)^2\leq 
w(R)\left(\frac{n}{m}\right)^2+a\gamma \left(\frac{n}{m}\right)^2\leq \left(\frac{4}{13}+\frac{2\gamma}{3}\right)n^2 \] edges to $G$. Combining with (\ref{nonedge of R}), we completed the proof of Theorem~\ref{theorem K_12 density}.

\subsection{$K_{13}$-free: Proof of Theorem~\ref{K_13 density}}

Let all the parameters be as in (\ref{parameter}). Let $G$ be an $n$-vertex $K_{13}$-free graph with $\alpha(G)\leq \delta \Q(5,n)$ and $R\coloneqq R(\epsilon, \gamma/2)$ be the corresponding cluster graph. Our aim now is to prove \[
e(G)\leq \frac{4}{15}n^2+\gamma n^2.\] Recall that by the $p=3$ version of Lemma~\ref{lemma to derive K_4 free}, $R$ is $K_5$-free. Similarly as in Sections~\ref{9pf}, \ref{9-11pf} and \ref{12pf}, we will first prove that $R$ contains no $K_4$ and then analyze the triangles in $R$.

\begin{claim}\label{chubby K4}
	$R$ is $K_4$-free.
\end{claim}

\begin{proof}
	The proof is similar to the proof of Claim~\ref{cluster_k4-free} where we apply the dependent random choice method. The only difference is the choice of $c$ and $t$ in the application of Lemma~\ref{lemma 3.1 4-partite generalized}.
	
	Suppose that $\{1,2,3,4\}$ spans a $K_4$ in $R$. Then, for each $i \in [4]$, there exists a subset $V_i^* \subseteq V_i$ such that $|V_i^*| = (1-3\epsilon)|V_i|$ and $\delta^{cr}(G[V_1^*,V_2^*,V_3^*, V_4^*])\geq \gamma|V_i^*|/4$. Applying Lemma~\ref{lemma 3.1 4-partite generalized} to $G[V_1^*,V_2^*,V_3^*,V_4^*]$ with $k=4, c=1/5$ and $t=3$ gives us a set $S \subseteq V_1^*$ of size $\frac{1}{2}|V_1^*|^{4/5}\geq \frac{1}{3}(\frac{n}{m})^{4/5}$ such that every triple $P \in \binom{S}{3}$ satisfies $|N(P,V_i^*)| \geq (\frac{\gamma}{4})^{120}\frac{n}{m}\geq \gamma^{121}\frac{n}{m}$ for each $i \in \{2,3,4\}$. Recall that $\alpha(G)\leq \delta \Q(5,n)$ and $\Q(5,n)=O(n^{1/3})$, so $\frac{1}{3}(\frac{n}{m})^{4/5}> (\alpha(G))^2$. Therefore, $S$ contains a triangle $uvw$ with $|N(\{u,v,w\},V_i^*)|\geq \gamma^{121}\frac{n}{m}$ for each $i \in \{2,3,4\}$.
	
	By applying Lemma~\ref{slicing lemma} and deleting all vertices of low degree if necessary, we could get subsets $V_i' \subseteq N(\{u,v,w\},V_i^*)$ for $i \in \{2,3,4\}$ satisfying that $|V'_2|=|V'_3|=|V'_4| \geq \gamma^{123}\frac{n}{m}$, $\delta^{cr}(G[V'_2,V'_3,V'_4]) \geq  \gamma|V'_i|/5$, and $(V'_i,V'_j)$ is $(\sqrt{\epsilon}, \gamma/4)$-regular for every pair $\{i,j\}\in \binom{\{2,3,4\}}{2}$. We apply Lemma~\ref{lemma 3.1 4-partite generalized} to $G[V'_2,V'_3,V'_4]$ with $k=3, c=1/5$ and $t=3$. This gives us a set $S' \subseteq V'_2$ of size $\frac{1}{2}|V'_2|^{4/5}\geq \gamma^{124}(\frac{n}{m})^{4/5}$ such that every triple $P \in \binom{S'}{3}$ satisfies $|N(P,V'_i)| \geq (\frac{\gamma}{5})^{80}|V'_i|\geq \gamma^{204}\frac{n}{m}$ for each $i \in \{3,4\}$. Since $\gamma^{124}(\frac{n}{m})^{4/5}>(\alpha(G))^2$, $S'$ contains a triangle $xyz$ with $|N(\{x,y,z\},V'_i)|\geq \gamma^{204}\frac{n}{m}$ for each $i \in \{3,4\}$. 
	
	Again, by applying Lemma~\ref{slicing lemma} and deleting all vertices of low degree if necessary, we could get $V_i'' \subseteq N(\{x,y,z\},V'_i)$ for $i \in \{3,4\}$ such that $|V''_3|=|V''_4| \geq \gamma^{206}\frac{n}{m}$, $\delta(G[V''_3,V''_4])\geq \gamma|V''_i|/6$ and $(V''_3,V''_4)$ is $(\epsilon^{1/4}, \gamma/5)$-regular. We apply Lemma~\ref{lemma 3.1 4-partite generalized} once more to $G[V''_3,V''_4]$ with $k=2, c=1/5$ and $t=3$. This gives us a set $S''\subseteq V''_3$ of size $\frac{1}{2}|V''_3|^{4/5}\geq \gamma^{207}(\frac{n}{m})^{4/5}$ such that every triple $P \in \binom{S''}{3}$ satisfies $|N(P,V''_4)| \geq (\frac{\gamma}{6})^{40}|V''_3|\geq \gamma^{247}\frac{n}{m}$. Since $\gamma^{207}(\frac{n}{m})^{4/5}>(\alpha(G))^2$, $S''$ contains a triangle $abc$ with $|N(\{a,b,c\},V''_4)| \geq \gamma^{247}\frac{n}{m}$. Note that $\Q\left(5,\frac{\gamma^{247}}{m}n\right)>\delta\Q(5,n)\geq \alpha(G)$. Therefore, $N(\{a,b,c\},V''_4)$ contains a $K_5$, which together with $uvw, xyz$ and $abc$ forms a $K_{14}\supseteq K_{13}$, a contradiction.
\end{proof}

Recall that a triangle $ijk$ in $R$ is \textit{$(4/5+\gamma)$-chubby} if $d_G(V_{i'}, V_{j'})>4/5 + \gamma$ for some $i'j' \in \binom{\{i,j,k\}}{2}$.

\begin{claim}\label{K_12-free chubby triangle}
No triangle in $R$ is $(4/5+\gamma)$-chubby.
\end{claim}

\begin{proof}
	The proof is almost the same as of Claim~\ref{modified Claim 3.2a}, with $\Q(4,\cdot)$ replaced by $\Q(5,\cdot)$. Thus we could find a triangle in $V^*_1$ and two copies of $K_5$ in $V'_2, V'_3$ respectively, together forming a $K_{13}$, which leads to a contradiction.
\end{proof}

Let $a$ be the number of edges contained in some triangle, then $a\leq m^2/2$. Let $b\coloneqq e(R)-a$. By Claim~\ref{chubby K4} and Theorem~\ref{K_5-free_chubby}, \[
a\cdot \frac{4}{5} + b \leq \left(\frac{4}{15}+\frac{\gamma}{6}\right) m^2.
\] By Claim~\ref{K_12-free chubby triangle}, every edge $ij\in E(R)$ contained in some triangle satisfies that $d_G(V_i, V_j)<4/5+\gamma$. Therefore, $E(R)$ contributes at most \begin{equation*}
		a\left(\frac{4}{5}+\gamma\right)\left(\frac{n}{m}\right)^2+b\left(\frac{n}{m}\right)^2\leq\left(\frac{4}{15}+\frac{\gamma}{6}+\frac{a\gamma}{m^2}\right)n^2 \leq \left(\frac{4}{15}+\frac{2}{3}\gamma\right)n^2
\end{equation*} edges to $G$. Combining with (\ref{nonedge of R}), we completed the proof of Theorem~\ref{K_13 density}.

\subsection{An improved result for $K_{13}$-free case}\label{subsec13}

Now we prove Theorem~\ref{K_13 with assumption density} by assuming that Conjecture~\ref{Ramsey conj} holds for $\ell=5$. 
We only need to slightly modify the proof of Theorem~\ref{K_13 density}. Let all the parameters be as in (\ref{parameter}). Let $G$ be a $K_{13}$-free graph on $n \geq n_0$ vertices with $\alpha(G) \leq \delta \Q(5,n)$ and $R\coloneqq R(\epsilon, \gamma/2)$ be the corresponding cluster graph. Since we assume that Conjecture~\ref{Ramsey conj} holds for $\ell=5$, $R$ is $K_4$-free by the $p=4$ version of Lemma~\ref{lemma to derive K_4 free}. By an argument similar to the proof of Claim~\ref{chubby K4}, we have the following claim.
    
    \begin{claim}
    	$R$ is triangle-free.
    \end{claim}

    \begin{proof}
	Since we assume the $\ell=5$ case of Conjecture~\ref{Ramsey conj} is true now, instead of finding a triangle in $V_1$, we find a $K_4$ in $V^*_1$. By applying the dependent random choice method one more time, we can still find a $K_4$ and a $K_5$ in $V'_2, V'_3$ respectively, which would force to have a $K_{13}$ in $G$, leading to a contradiction.
\end{proof}

    Now we apply Theorem~\ref{cliqueweighting} instead of Theorem~\ref{K_5-free_chubby} and get the required result.

\section{Concluding Remarks}

In our paper, we considered the following general problem:

\begin{problem}\label{prob}
	Given an integer $s$ and a function $f(n)$, what is $\RT(n, K_s, f(n))$, i.e.~the maximum number of edges in an $n$-vertex $K_s$-free graph with independence number at most $f(n)$?
\end{problem}

One obstacle to solving Problem~\ref{prob} is Conjecture~\ref{Ramsey conj}: it is not known if there is a jump of order $n^c$ between different off-diagonal Ramsey numbers, where $c$ is a constant. Conditioning on that Conjecture~\ref{Ramsey conj} holds, Theorem 3.9 in~\cite{BHS} answers Problem~\ref{prob} when $f(n)=\Q(p,n)$.

\begin{theorem}[\cite{BHS}, Theorem 3.9]\label{bhs-general}
	If $r=\lfloor\frac{s-1}{p-1}\rfloor$ and Conjecture~\ref{Ramsey conj} holds for $\ell=p$, then \[\rho\tau(K_s,\Q(p,n))=\frac{1}{2}\left(1-\frac{1}{r}\right).\]
\end{theorem}

The extremal graph in Theorem~\ref{bhs-general} is obtained from a balanced complete $r$-partite graph by replacing each of the $r$ parts by a $K_p$-free graph with independence number $\Q(p,n/r)$. Given integers $s$ and $p$, if $\lfloor\frac{s-1}{p-1}\rfloor=\lfloor\frac{s-1}{p}\rfloor$, then it follows from Theorem~\ref{bhs-general} that $K_s$ has no phase transition at $\Q(p,n)$ assuming Conjecture~\ref{Ramsey conj} holds. The complications of determining the existence of phase transitions come when the parameters $s$ and $p$ do not satisfy $\lfloor\frac{s-1}{p-1}\rfloor=\lfloor\frac{s-1}{p}\rfloor$. The first such instance is $s=4, p=2$. It was resolved in~\cite{BollErdos} using the Bollob\'as-Erd\H os graph, implying that $K_4$ has a phase transition at $f(n)=n/3$. More generally, the Bollob\'as-Erd\H os graph was used for the constructions in the cases where $s=2r, p=2$, and the extremal graph in such a case is obtained from a complete $r$-partite graph by changing one pair of classes to the Bollob\'as-Erd\H os graph. For larger $p$, we do not have such constructions, and the current methods for proving upper bounds are not expected to prove that no such construction using the Bollob\'as-Erd\H os graph exists.


We think that it would be extremely interesting to find such constructions. A good
first step toward this would be to predict that what type of constructions would be the most
useful for our problem. This could potentially be done by creating some weighted Tur\'an-type
problems, whose solutions would be useful for such approaches. Such steps were initiated in our paper: Theorem~\ref{cliqueweighting}, Theorem~\ref{K_5-free_chubby}, and Theorem~\ref{K_4-free-heavy} were conjectured by us at the early stage in our project, and were solved in~\cite{BL},~\cite{CliqueWeighting} and~\cite{CliqueWeighting2}. It is worth mentioning that weighted Tur\'an-type results are interesting by their own and could be applied to some other problems. More information about them could be found in \cite{BL}.

In the rest of this section, we state two general propositions, which could be proved by our methods. However, we are still far from fully answering Problem~\ref{prob}, even with the assumption that Conjecture~\ref{Ramsey conj} holds.

Conditional on Conjecture~\ref{Ramsey conj}, we first give the following generalization of Theorem~\ref{theorem K_9,10,11 density} by an argument similar to Section~\ref{9-11pf}.

\begin{proposition}\label{generalprop}
	If $2p+1\leq s\leq 2p+3$ and Conjecture~\ref{Ramsey conj} holds for $\ell=p\geq 4$, then \[
	\rho\tau(K_s,o(\Q(p,n)))= \frac{1}{4}.\]
\end{proposition}

\noindent\textbf{Remark. }When $p=4$, we get Theorem~\ref{theorem K_9,10,11 density}. Theorem~\ref{K_13 with assumption density} is a special case when $p=5$ and $s=2p+3=13$. If $p\geq 6$, then $r\coloneqq\lfloor\frac{s-1}{p-1}\rfloor=2$. Combining with Theorem~\ref{bhs-general}, we conclude that $K_s$ has no phase transition at $\Q(p,n)$ if $2p+1\leq s\leq 2p+3$ and Conjecture~\ref{Ramsey conj} holds for $\ell=p\geq 4$.\\

\noindent\textit{Proof of Proposition~\ref{generalprop}.} For the upper bound, the proof idea is the same as in the proof of Theorem~\ref{theorem K_9,10,11 density} in Section~\ref{9-11pf} and we only give a sketch here. Let $G$ be an $n$-vertex $K_s$-free graph with $\alpha(G)\leq \delta \Q(p,n)$, where $2p+1\leq s\leq 2p+3$. Let $R\coloneqq R(\epsilon,\gamma/2)$ be the corresponding cluster graph after applying Szemer\'{e}di's Regularity Lemma to $G$. Since we assume that Conjecture~\ref{Ramsey conj} holds for $\ell=p\geq 4$, we can apply Lemma~\ref{lemma to derive K_4 free} and conclude that $R$ is $K_4$-free. 
	By Theorem~\ref{cliqueweighting}, it suffices to show that no triangle in $R$ is $(3/4+\gamma)$-chubby, where the definition of \emph{$(3/4+\gamma)$-chubby} is the same as in Section~\ref{9-11pf}. Suppose for a contradiction that $\{1,2,3\}$ spans a $(3/4+\gamma)$-chubby triangle in $R$ with $d(V_2,V_3)>3/4+\gamma$. We follow the notation in the proof of Claim~\ref{modified Claim 3.2a} and use the subsets $V^\ast_1, V^\ast_2, V^\ast_3$ with high minimum crossing degree. Apply Lemma~\ref{lemma 3.1 4-partite generalized} to $G[V^\ast_1, V^\ast_2, V^\ast_3]$ with $k=3$ and $t=p-1$. Since we assume that Conjecture~\ref{Ramsey conj} holds for $\ell=p$, instead of finding a $K_3$ in $S\subseteq V^\ast_1$, we could find a $K_{p-1}$ whose vertices have linearly many common neighbors in both $V^\ast_2$ and $V^\ast_3$. As $\alpha(G)\leq \delta \Q(p,n)$ and $d(V'_2,V'_3)>3/4$, we could find a $K_4$ in $V'_2$, say $T$, and a $K_p$  in $N(T,V'_3)$, which implies that $G$ contains a $K_{2p+3}\supseteq K_s$, a contradiction.
	
	For the lower bound, let $H$ be a $K_{p+1}$-free graph on $n/2$ vertices with independence number $o(\Q(p,n))$ and with $o(n^2)$ edges. Such a graph exists by taking $q=p+1$ in Theorem~\ref{bhs-general}. Let $G$ be obtained from the union of two vertex-disjoint copies of $H$, say $A$ and $B$, by joining every vertex in $A$ to every vertex in $B$. Then, $G$ is $K_{2p+1}$-free, thus $K_s$-free for $2p+1\leq s\leq 2p+3$, with $n^2/4+o(n^2)$ edges and $\alpha(G)\leq o(\Q(p,n))$. \qed\\

Notice that in all proofs of the main results, we only applied Lemma~\ref{lemma to derive K_4 free} with $q\leq 5$. Since Lemma~\ref{lemma to derive K_4 free} holds for all $p,q\geq 2$, it seems natural to generalize our method by applying Lemma~\ref{lemma to derive K_4 free} with larger $p,q$ and try to solve Problem~\ref{prob} for larger $s$, with $f(n)=o(\Q(p,n))$ for some $p<s$. Unfortunately, we would encounter obstacles here. When $s\leq (p-1)q$ with $q\geq 6$, although we could assume that Conjecture~\ref{Ramsey conj} holds for $\ell=p$ and thus forbid $K_q$ in the cluster graph $R$ by Lemma~\ref{lemma to derive K_4 free}, the information we obtain from the large cliques contained in $R$, say $K_{q-1},\ldots,K_5$, is not as useful as we thought due to the existence of $K_3$ in $R$. To be more specific, in order to generalize Theorem~\ref{theorem K_9,10,11 density} with the application of Theorem~\ref{cliqueweighting}, it is necessary to show that no triangle in $R$ is $(3/4+\gamma)$-chubby. However, by the proof of Proposition~\ref{generalprop}, this would require $s\leq 2p+3$, thus reduce the case $s\leq (p-1)q$ to the condition of Proposition~\ref{generalprop}. This again implies that weighted Tur\'an-type results would be beneficial to our approach.

Assume that Conjecture~\ref{Ramsey conj} holds for $\ell=t\geq 4$. By applying Lemma~\ref{lemma to derive K_4 free} with $q=4$ or the dependent random choice method, we have the following general result, which together with Theorem~\ref{bhs-general} implies that $K_s$ has no phase transition at $\Q(t,n)$ if $3t+1\leq s\leq 4t-4$ and $K_s$ has phase transition at $\Q(t,n)$ if $s=4t-3$. Notice that the case $t=4$, which implies that $s=13$, has been proved in~\cite{BL}.

\begin{proposition}
	If $3t+1\leq s\leq 4t-3$ and Conjecture~\ref{Ramsey conj} holds for $\ell=t\geq 4$, then \[
	\rho\tau(K_s,o(\Q(t,n)))= \frac{1}{3}.\]
\end{proposition}

\begin{proof}
	For the upper bound, the proof idea is the same as in the proofs of the main results in Section~\ref{proofofmain} and we only give a sketch here. Let $G$ be an $n$-vertex $K_s$-free graph with $\alpha(G)\leq \delta \Q(t,n)$, where $3t+1\leq s\leq 4t-3$. Let $R$ be the corresponding cluster graph on $m$ vertices after applying Szemer\'{e}di's Regularity Lemma to $G$. Since we assume that Conjecture~\ref{Ramsey conj} holds for $\ell=t$, we can show that $R$ is $K_4$-free by applying Lemma~\ref{lemma 3.1 4-partite generalized} three times as in the proof of Claim~\ref{chubby K4}. Therefore, $E(R)$ contributes at most $T(m,3)(n/m)^2=n^2/3$ edges to $G$. 
	
	For the lower bound, let $H$ be a $K_{t+1}$-free graph on $n/3$ vertices with $\alpha(H)= o(\Q(t,n))$ and with $o(n^2)$ edges. Such a graph exists by taking $q=t+1$ in Lemma~\ref{bhs-general}. Let $G$ be obtained from the union of three vertex-disjoint copies of $H$, say $A_1, A_2, A_3$, by joining every vertex in $A_i$ to every vertex in $A_j$ for every $ij\in\binom{[3]}{2}$. Then, $G$ is $K_s$-free for $s\geq 3t+1$ with at least $n^2/3$ edges.
\end{proof}

\section{Acknowledgment}

The authors are grateful to the anonymous referees for their careful reading and valuable comments.

\bibliographystyle{abbrv}
\bibliography{reference}

\end{document}